\documentclass[12pt, twoside, leqno]{amsart}
\usepackage{amsmath,amsthm}
\usepackage{amssymb,latexsym}
\usepackage{graphicx}
\usepackage{xcolor}
\usepackage{tikz}
\theoremstyle{plain}
\newtheorem{thm}{Theorem}[section]
\newtheorem{corollary}[thm]{Corollary}
\newtheorem{lemma}[thm]{Lemma}

\newtheorem{prop}[thm]{Proposition}
\theoremstyle{definition}

\newtheorem*{testmassive}{Test of massiveness}

\newtheorem{example}[thm]{Example}
\newcommand{\ud}{\mathrm{d}}

\newcommand{\gP}{\mathbb{P}}

\newcommand{\gN}{\mathbb{N}}
\newcommand{\gZ}{\mathbb{Z}}

\newcommand{\cP}{\mathcal{P}}
\newcommand{\cA}{\mathcal{A}}
\newcommand{\cE}{\mathcal{E}}
\newcommand{\cF}{\mathcal{F}}
\newcommand{\cT}{\mathcal{T}}
\newcommand{\cL}{\mathcal{L}}

\newcommand{\cM}{\mathcal{M}}
\newcommand{\cN}{\mathcal{N}}

\newcommand{\kap}{\operatorname*{Cap_\alpha}}
\numberwithin{equation}{section}
\begin{document}
\title{On massive sets \\for subordinated random walks}
\author[A. BENDIKOV]{ALEXANDER BENDIKOV}
	\thanks{Research of A. Bendikov was supported by National Science Centre, Poland, Grant DEC-
2012/05/B/ST1/00613}
\author[W. CYGAN]{WOJCIECH CYGAN}
	\thanks{Research of W. Cygan was supported by National Science Centre, Poland, Grant DEC-
2013/11/N/ST1/03605 and by SFB 701 of German Research Council}
\address{A. Bendikov \\ Institute of Mathematics\\
	Wroc\l{}‚aw University\\
	50-384 Wroc\l{}aw, Pl. Grunwaldzki 2/4, Poland}
\email{bendikov@math.uni.wroc.pl}
\address{W. Cygan \\ Institute of Mathematics\\
	Wroc\l{}aw University\\
	50-384 Wroc\l{}aw, Pl. Grunwaldzki 2/4, Poland}
\email{cygan@math.uni.wroc.pl}
\begin{abstract}
We study massive (reccurent) sets with respect to a certain random walk
$S_\alpha $ defined on the integer lattice $\mathbb{Z} ^d$, $d=1,2$. Our random walk
$S_\alpha $ is obtained from the simple random walk $S$ on $\mathbb{Z} ^d$ by the procedure
of discrete subordination. $S_\alpha $ can be regarded as a discrete
space and time counterpart of the symmetric $\alpha $-stable L\'{e}vy process in
$\mathbb{R}^d$.
 In the case $d=1$ we show that some remarkable proper subsets of $\mathbb{Z}$ , e.g.
the set $\mathcal{P}$ of primes, are massive whereas some proper subsets of $\mathcal{P}$ 
such as Leitmann primes $\mathcal{P}_h$ are massive/non-massive depending on
the function $h$. Our results can be regarded as an extension of the
results of McKean (1961) about massiveness of the set of primes for the
simple random walk in $\mathbb{Z}^3$.
 In the case $d=2$ we study massiveness of thorns and their proper subsets.
The case $d>2$ is presented in the recent paper
\textsc{Bendikov and Cygan} \cite{Cygan}.
\end{abstract}
\maketitle
\tableofcontents  
\renewcommand{\thefootnote}{}


\renewcommand{\thefootnote}{\arabic{footnote}}
\setcounter{footnote}{0}
\section{Introduction}
\par The purpose of this paper is to study massive (recurrent) sets with respect to a certain class of random walks on the integer lattice $\gZ ^d$ which are driven by low moment measures. Recall that some exhaustive results about massive sets with respect to random walks having finite second moment have been obtained in the middle of the last century, see for instance \textsc{Spitzer} \cite{spitzer}, \textsc{It\^{o} and McKean} \cite{ItoMcKean}, \textsc{Doney} \cite{doney}, \textsc{McKean} \cite{McKean}, \textsc{Dynkin and Yuskhevich} \cite{dynkin}.  
\par Perhaps the simplest way to build a low moment random walk is to use the Bochner's idea of subordination (random change of time). Subordination has been used successfully in the context of continuous time Markov processes. Recently the idea of subordination has been used by Bendikov and Saloff-Coste \cite{lscbendikov} in the context of discrete time Markov chains.
\par We recall briefly the construction of a subordinated random walk form \cite[Section 2]{lscbendikov}. Let a random walk $X=(X_n)_{n\geq 0}$ with the state space $\gZ ^d$ be given. Let $\tau =(\tau _n)_{n\geq 0}$ be a random walk on $\gZ ^+$. We assume that $\tau $ and $X$ are independent. The subordinated random walk $Y=(Y_n)$ is defined as $ Y_n = X_{\tau _n}$. Notice that even if $X$ has finite second moment the subordinated random walk $Y$ may well have infinite second moment. That is what happens in the basic example of the paper \textsc{Bendikov and Cygan} \cite{Cygan}: $X$ is the symmetric simple random walk in $\gZ ^d$ (denoted by $S$) and $\tau $ is a discrete version of the classical $\alpha /2$-stable subordinator, $0<\alpha <2$. In this case $Y$ is called the \textit{$\alpha$-stable random walk} and is denoted $S_\alpha$, see \cite[Section 1, Definition 1.2]{Cygan} for details. If we denote by $P$ the transition operator of the random walk $S$, then by \cite[Proposition 2.3]{lscbendikov} the transition operator $P_\alpha$ of $S_\alpha $ is
\begin{align*}
P_\alpha = I -(I-P)^{\alpha /2}.
\end{align*}
Moreover, if $p(n,x)$ is the transition function of the symmetric simple random walk $S$, \cite[Proposition 2.3]{lscbendikov} implies that the transition function $p_\alpha (n,x)$ of the $\alpha $-stable random walk has the form
\begin{align*}
p_\alpha (n,x) = \sum _{k=1}^\infty  p(k,x)\gP (\tau _n=k).
\end{align*} 
The Green function $G_{\alpha}(x,y)$ and the Green potential $G_{\alpha}f(x)$,
$f \geq 0$, are defined as
\begin{align*}
G_{\alpha}(x,y) = \sum_{n\geq 0} p_{\alpha}(n,x-y)
\end{align*}
and
\begin{align*}
G_{\alpha}f(x) =\sum_{y\in Z^d} G_{\alpha}(x,y) f(y).
\end{align*}
According to \cite[Theorem 2.3]{Cygan}, $G_{\alpha}(x,y)< \infty $ for all $x,y \in
\gZ ^d$, i.e. $S_{\alpha}$ is transient, if and only if $0<\alpha<d$. In
particular, when $d\geq 2$, $S_{\alpha}$ is transient for all $0<\alpha<2$.
\par Assume that $S_{\alpha}$ is transient. The capacity $\mathrm{Cap}_{\alpha}(B)$ is
defined as
\begin{align*}
\mathrm{Cap}_{\alpha}(B)= \sum_{b\in B} \phi_{B}(b),
\end{align*}
where $\phi_{B}$ is the equilibrium distribution on $B$, that is, $\phi_{B}\geq
0$, $\mathrm{supp}(\phi_{B})\subseteq B$ and the potential $G_{\alpha}\phi_{B}$ is less or
equal $1$ everywhere and equals $1$ on $B$. We refer to \cite[Chapter VI]{spitzer} 
and \cite[Chapter I]{dynkin} for the general treatment of
capacities defined by transient random walks.
\par By \cite[Chapters IX and X]{meyer}, the cone of potentials $\{G_{\alpha}f\}$ defined
by the discrete time Markov semigroup $\{P_{\alpha}^n\}_{n\in \gN}$ coincides
with the cone of potentials defined by the continuous time Markov
semigroup $P_{\alpha}^t=\exp[-t(I-P_{\alpha})], t>0$. Indeed, for any
$\phi\geq 0$, we have
\begin{align*}
G_{\alpha}\phi(x)=\int_0^{\infty} P_{\alpha}^t\phi(x) dt, 
\end{align*}
for all $x\in Z^d$. In particular, $\mathrm{Cap}_{\alpha}(B)$ coincides with the capacity
defined by the continuous time Markov semigroup, see \cite[Chapter VI]{Blumen}.
\par The semigroup $\{P_{\alpha}^t\}_{t>0}$ acts in $L^2=L^2(\mathbb{Z}^d,\nu)$, $\nu $ is the
counting measure, and it is symmetric and Markovian. Its $L^2$--generator is
$I-P_{\alpha}$. According to \cite[Chapter 2]{Fuku}, $\mathrm{Cap}_{\alpha}(B)$
coincides with $L^2$--capacity defined by the Dirichlet form
$\mathcal{E}_{\alpha}(f,f)=\int (f-P_{\alpha}f)f d\nu$. Thus finally, the capacity
$\mathrm{Cap}_{\alpha}(B)$ defined originally by the random walk $S_{\alpha}$ coincides
with the capacity defined by the Dirichlet form $\mathcal{E}_{\alpha}(f,f)$. We use this
correspondence later in Section 4 to prove certain lower bounds of
capacities defined by subordinated random walks.
\par Assume that $S_{\alpha}$ is transient. Let $B$ be a proper subset of $\gZ ^d$. Let
$p_{B}(x)$ be the hitting probability of $B$. The set $B$ is called massive
(recurrent) if $p_{B}(x)=1$ for all $x\in \gZ ^d$ and non-massive otherwise. One
of the main ingredients in our study is the following test of
massiveness, see \cite[Theorem 3.1]{Cygan}.
\begin{testmassive}
A subset $B\subset \gZ ^d$ is $S_\alpha $-massive if and only if  
\begin{align*}
\sum _{n\geq 1}\frac{\kap(B_n)}{2^{n(d-\alpha )}}=\infty ,
\end{align*}
where $B_n=\{b\in B:\, 2^n\leq \Vert b\Vert <2^{n+1}\} $.
\end{testmassive}
\par In the paper \cite{Cygan} we concentrate ourselves on the case $d\geq 3$ and $0<\alpha < 2$. In this setting any cone is a massive set. Hence the problem becomes non-trivial when we consider thin sets such as thorns. Let $t(n)$ be a non-decreasing sequence of positive numbers such that $t(n)=o(n)$ at infinity. We define the thorn $\mathcal{T}$ as 
\begin{align*}
\mathcal{T}=\{(x_1,\ldots ,x_d)\in \gZ ^d: \Vert (x_1,\ldots ,x_{d-1})\Vert \leq t(x_d),\, x_d\geq 1  \},
\end{align*} 
where $\Vert \cdot \Vert $ denotes the Euclidean norm.
Let $S$ be the simple random walk. When the dimension $d$ is $3$ the set
\begin{align*}
\{x=(x_1,x_2,x_3)\in \gZ ^3:\, x_1=x_2=0\}
\end{align*}
is massive with respect to $S$. In particular any thorn is massive in $\gZ ^3$. Thus $S$-massiveness of the thorns becomes non-trivial when the dimension $d\geq 4$. This problem has been completely solved by It\^{o} and McKean in the celebrated paper \cite{ItoMcKean}. 
\begin{thm}\cite[Section 6]{ItoMcKean}
Let $S$ be the symmetric simple random walk. Assume that $d\geq 4$.
Then  $\mathcal{T}$ is $S$-massive if and only if the following condition holds 
\begin{align*}
\sum _{n>0}\Big(\frac{t(2^n)}{2^n}\Big)^{d-3}=\infty .
\end{align*}
\end{thm}  
The main result of the paper \cite{Cygan} is the following statement. 
\begin{thm}\label{thm2}\cite[Theorem 4.4]{Cygan}
Let $S_\alpha $ be the $\alpha$-stable random walk. Assume that $d\geq 3$.
Then  $\mathcal{T}$ is $S_\alpha$-massive if and only if the following condition holds 
\begin{align*}
\sum _{n>0}\Big(\frac{t(2^n)}{2^n}\Big)^{d-\alpha -1}=\infty .
\end{align*}
\end{thm}
Thus Theorem \ref{thm2} can be seen as an extension of the theorem of It\^{o} and McKean.
In this paper we study $S_\alpha$-massive sets assuming that $d\leq 2$ and $0<\alpha <d$. In this case $S_\alpha $ is transient. 
\par In Section 2 we consider subordinated random walks in $\gZ $ and study massive sequences $\{b_n\}\subset \gZ $, e.g. the set $\mathcal{P}$ of primes and some of its subsets such as Piatetski-Shapiro primes etc. The results obtained in this section are in spirit of the papers of McKean \cite{McKean} and Bucy \cite{bucy}. 
\par In Section 3 we consider $\alpha$-stable random walks in $\gZ ^2$. When $0<\alpha <1$ the set $A=\{x=(x_1,x_2)\in \gZ ^2:\, x_2=0\}$ is not massive, hence we study massive thorns. When $1\leq \alpha <2$ the set $A$ is massive, hence we study massive subsequences of $A$. In both cases we also study massive subthorns, i.e. the sets of the form
\begin{align*}
\mathcal{V}=\mathcal{T}\cap \{(x_1,x_2)\in \mathbb{Z} ^2:\, x_2\in \mathcal{A}\},
\end{align*}
where $\mathcal{A}$ is a given sequence of integers.
\par One of the main ingredient in our proofs is a very precise lower bound of the $S_\alpha$-capacity. A proof of this bound we provide in the concluding Section 4.\\
\textbf{Notation.}
For any two non-negative functions $f$ and $g$, $f(r)\sim g(r)$ at $a$ means that $\lim _{r\rightarrow a}f(r)/g(r)=1$; $f(x)=O(g(x))$ if $f(x)\leq Cg(x)$, for some constant $C>0$, and $f(x)\asymp g(x)$ if $f(x)=O(g(x))$ and $g(x)=O(f(x))$. We also write $f(r)=o(g(r))$ at $a$ if $\lim _{r\rightarrow a}f(r)/g(r)=0$. 
\section{Massive subsets of $\gZ $}
\par Let $S_\alpha$, $0<\alpha <1$, be the subordinated random walk in $\mathbb{Z}$ as defined above. $S_\alpha $ is transient whence any finite subset of $\gZ$ is not massive whereas the whole of $\gZ$ is evidently massive. We study here proper infinite subsets of $\gZ$, for instance the set $\cP$ of primes.
\par 
One of the main ingredient in our proofs is the asymptotic of the Green function $G_\alpha (x,y)$ obtained in the paper \cite[Theorem 2.4]{Cygan},
\begin{align}
G_\alpha (x,y)\sim  \frac{\, \, 2^{-\alpha /2}\pi ^{-1/2}}{\Gamma \big( \frac{\alpha}{2}\big) }\, \Gamma \Big( \frac{1-\alpha}{2}\Big)\vert x-y\vert ^{\alpha -1},\quad \textrm{as}\ \ \vert x-y\vert \to \infty .\label{Green1}
\end{align}
The second important ingredient is the following lower bound of the capacity $\kap (B)$,
\begin{align}
\kap (B)\geq c|B|^{1-\alpha},\label{isoperim}
\end{align}
where $|B|$ is the cardinality of the set $B$ and $c=c(\alpha ) >0$ is a constant. The proof of the inequality (\ref{isoperim}) will be provided in the concluding Section 4, Corollary \ref{CorPsi}.
\par An increasing sequence $\mathcal{A}=\{ a_n : n\in \gN\}$ is called \textit{superlinear} if it satisfies 
\begin{align}
a_n\geq a_{n-k}+a_k,\quad \textrm{for all}\ \ 0<k<n .\label{superlin}
\end{align}
Examples of superlinear sequences $\{a_n\}$ are: $[n^\beta]$, $\beta \geq 1$; $[n\log n]$. The proof of the next statement is similar to that of \textsc{Bucy} \cite[Corollary 4.1]{bucy} but some details and the applications we have in mind require adjustments and variations.
\begin{thm}\label{thmSuperlin}
Let $\mathcal{A}=\{ a_n : n\in \gN\}$ be superlinear. Then the set $\mathcal{A}$ is massive if and only if 
\begin{align}
\sum _{n\geq 1}\frac{1}{a_n^{1-\alpha}}=\infty .\label{ser1}
\end{align}
\end{thm}
\begin{proof}
Suppose that $\sum _{n\geq 1}a_n^{\alpha -1}<\infty$. Then by the equation (\ref{Green1}), 
\begin{align*}
\sum _{n\geq 1}G_\alpha (0,a_n)<\infty .
\end{align*}
Recall that $G_\alpha (0,a_n)$ is the expected number of visits to $a_n$ of $S_\alpha $ started at $0$. Applying the Borel-Cantelli lemma we obtain non-massiveness of $\mathcal{A}$.
\par Suppose that $\sum _{n\geq 1}a_n^{\alpha -1}=\infty$ and $\mathcal{A}$ is not massive. Let $\phi$ be the equilibrium measure of the set $\mathcal{A}$. For each $n\geq 1$  we have
\begin{align}
\sum _{m\geq 1}G_\alpha (a_m,a_n)\phi (a_m)=1.\label{Hopf}
\end{align} 
Let $F_N(m)$ be defined as
\begin{align*}
F_N(m)=\frac{\sum _{n=1}^Na_n^{\alpha -1 }G_\alpha (a_m,a_n)}{\sum _{n=1}^{N}a_n^{\alpha -1}}.
\end{align*}
We claim that
\begin{align}
\sum _{m=1}^\infty F_N(m)\phi (a_m)=1\label{eqFN}
\end{align}
and
\begin{align}
\lim _{N\rightarrow \infty} F_N(m)=0.\label{LIMFN}
\end{align}
The equation (\ref{eqFN}) follows from the very definition. Since $a_n\rightarrow \infty $, by the Green-function asymptotic (\ref{Green1}), for any fixed $m\geq 1$ and for $N$ large enough there are some constants $c_1,c_2, c_3>0$ which depend only on $\alpha$ such that
\begin{align*}
F_N(m)&\leq \frac{c_1+c_2\sum _{n=m+1}^{N}a_n^{\alpha -1}(a_n-a_m)^{\alpha -1}}{\sum _{n=1}^{N}a_n^{\alpha -1} } \\
&\leq \frac{c_1}{\sum _{n=1}^{N}a_n^{\alpha -1}}+\frac{c_3\sum _{n=m+1}^{N}a_n^{2(\alpha -1)}}{\sum _{n=1}^{N}a_n^{\alpha -1} }.
\end{align*}
The first term tends to 0 by the assumption. Hence we are left to show that the second term tends to zero as well.
Let $\epsilon >0 $ be fixed. Since $a_n^{\alpha -1}\rightarrow 0$ we can find $M=M(\epsilon)$ such that $a_n^{\alpha -1}<\epsilon$ for all $n\geq M$. For any $N>M$ we have
\begin{align*}
\frac{\sum _{n=1}^{N}a_n^{2(\alpha -1)}}{\sum _{n=1}^{N}a_n^{\alpha -1} }&\leq \frac{\sum _{n=1}^{M}a_n^{2(\alpha -1)}+\epsilon \sum _{n=M+1}^{N}a_n^{\alpha -1}}{\sum _{n=1}^{N}a_n^{\alpha -1} }\\
&\leq \frac{C(M)}{\sum _{n=1}^{N}a_n^{\alpha -1}}+\epsilon
\end{align*}
where $C(M)>0$ depends only on $M$. This evidently proves (\ref{LIMFN}).
\par Now we use superlinearity of the sequence $\mathcal{A}$ and show that
\begin{align}
\lim _{N\rightarrow \infty}\sum _{m=1}^\infty F_N(m)\phi (a_m)=0.\label{eq34}
\end{align}
This will contradict (\ref{eqFN}) and the proof will be finished. \\
\textit{Claim.} For some $c=c(\alpha )>0$ and any $m$, $N$,
\begin{align}
F_N(m)\leq c(\alpha )a_m^{\alpha -1}.\label{FN1}
\end{align}
\par To prove (\ref{FN1}) we need the following two inequalities: 
\begin{align}
a_n^{\alpha -1}(a_m-a_n)^{\alpha -1}&\leq a_m^{\alpha -1}\big( a_n^{\alpha -1}+a_{m-n}^{\alpha -1} \big),\quad m>n,\label{seqone}\\
a_n^{\alpha -1}(a_n-a_m)^{\alpha -1}&\leq a_m^{\alpha -1}\big( a_n^{\alpha -1}+a_{n-m}^{\alpha -1} \big),\quad n>m\label{seqtwo}.
\end{align}
Since (\ref{seqone}) and (\ref{seqtwo}) can be proved similarly we prove (\ref{seqone}). By the assumption  $a_{m-n}\leq a_m-a_n$ we have
\begin{align*}
a_ma_n^{-1}&\leq a_ma_{m-n}^{-1}-a_na_{m-n}^{-1}+a_ma_n^{-1}-1\\
&=(a_m-a_n)( a_n^{-1}+a_{m-n}^{-1} ).
\end{align*}
It follows that
\begin{align*}
a_n^{-1}(a_m-a_n)^{-1}&\leq a_m^{-1}(a_n^{-1}+a_{m-n}^{-1}).
\end{align*}
Taking both sides to the power $1-\alpha $ and applying the inequality
\begin{align*}
(x+y)^{1-\alpha} <x^{1-\alpha} +y^{1-\alpha},\quad x,y>0
\end{align*}
we get (\ref{seqone}).
\par When  $m<N$ we have, for some  $c_1, c_2>0$ depending only on $\alpha $,
\begin{align}\label{fN}
F_N(m)\leq \frac{c_1\sum _{n=1, n\neq m}^Na_n^{\alpha -1}|a_m-a_n|^{\alpha -1}+c_2a_m^{\alpha -1} }{\sum _{n=1}^{N}a_n^{\alpha -1}}.
\end{align}
Substituting (\ref{seqone}) and (\ref{seqtwo}) in (\ref{fN}) 
we obtain (\ref{FN1}). For $m\geq N$ we use (\ref{seqone}):
\begin{align*}
F_N(m)\leq c_3\frac{a_m^{\alpha -1} \sum _{n=1}^N (a_n^{\alpha -1}+a_{m-n}^{\alpha -1} )}{\sum _{n=1}^{N}a_n^{\alpha -1}}
= c_3\, a_m^{\alpha -1} ( 1+ \Delta),
\end{align*}
for some $c_3=c_3(\alpha )>0$.
Since $a_{m-n}^{\alpha -1}\leq a_{N-n}^{\alpha -1}$ for $n<N\leq m$, we get
\begin{align*}
\Delta \leq \frac{\sum _{n=1}^{N-1}a_{N-n}^{\alpha -1}}{\sum _{n=1}^{N}a_n^{\alpha -1}}+\frac{a^{\alpha -1}_{m-N}}{\sum _{n=1}^{N}a_n^{\alpha -1}}
\leq 1+\frac{1}{\sum _{n=1}^{N}a_n^{\alpha -1}}\leq c.
\end{align*}
Applying now the claim we obtain that
\begin{align*}
F_N(m)\phi (a_m)\leq ca_m^{\alpha -1}\phi (a_m).
\end{align*}
The equations (\ref{LIMFN}), (\ref{Hopf}) and the dominated convergence theorem yield (\ref{eq34}).
\end{proof}
\begin{corollary}\label{convex}
Let $\{a_n\}$ be an increasing sequence of positive integers such that the sequence
$\Delta a_n=a_n-a_{n-1}$
is non-decreasing. Then $\{a_n\}$ is massive if and only if 
\begin{align*}
\sum _{n=1}^{\infty} a_n^{\alpha -1}=\infty .
\end{align*}
\end{corollary}
\begin{proof}
The proof of the corollary is identical to that of \textsc{Bucy} \cite[Sub-Corollary 4.1]{bucy}.
By the previous theorem it is sufficient to show that $a_n\geq a_{n-k}+a_k$. Set $a_0=0$. For any $k<n$ we have
\begin{align*}
a_n=\sum _{j=1}^n\Delta a_j=\sum _{j=1}^k\Delta a_j+\sum _{j=k+1}^n\Delta a_j\geq \sum _{j=1}^k\Delta a_j+\sum _{j=1}^{n-k}\Delta a_j=a_k+a_{n-k}.
\end{align*}
\end{proof}
In the following lemma we give useful estimates of the capacity $\kap (B)$.
\begin{lemma}
The capacity $\kap(B)$ of a set $B$ can be estimated as follows
\begin{align}
\frac{|B|}{\max _{a\in B}\sum _{b\in B}G_\alpha (a,b)}\leq \kap (B)\leq \frac{|B|}{\min _{a\in B}\sum _{b\in B}G_\alpha (a,b)}.\label{eq123}
\end{align}
\begin{proof}
Let $\phi _B$ be the equilibrium distribution for $B$, that is $G_\alpha \phi _B\leq 1$ and $G_\alpha \phi _B=1$ on $B$. We have
\begin{align*}
|B|&=\sum _{a\in B} \sum _{b\in B}G_\alpha (a,b)\phi _B(b)\\
&=\sum _{b\in B} \phi _B(b)\sum _{a\in B}G_\alpha (a,b)\leq \kap(B)\max _{b\in B}\sum _{a\in B}G_\alpha (a,b).
\end{align*}
Similarly
\begin{align*}
|B|\geq \kap(B)\min _{b\in B}\sum _{a\in B}G_\alpha (a,b)
\end{align*} 
and the proof is finished.
\end{proof}
\end{lemma}
\begin{example}
We show that the condition $a_n\geq a_{n-k}+a_k$ in the Theorem \ref{thmSuperlin} cannot be dropped. We adjust an example from \textsc{Bucy} \cite{bucy} to our setting. Let $A=\bigcup _{n\geq 1}A_n$, where 
\begin{align*}
A_n=\{ k\in \gN : 2^n\leq k<2^n(1+n^{-\gamma }) \}
\end{align*}
and $\gamma =\frac{2}{1-\alpha }$. The set $A=\{a_n\}$ does not satisfy the condition $a_n\geq a_{n-k}+a_k$. Since $|A_n|\rightarrow \infty$, for all $n$ large enough we can find $a_k\in A_n$ such that $a_{k-1}=a_k-1$. Then evidently $a_{k-1}+a_1> a_k$ because $a_1\geq 2$. We claim that
\begin{align*}
\sum _{n=1}^\infty \frac{1}{a_n^{1-\alpha}} =\infty ,
\end{align*}
while the set $A $ is not massive. Indeed, for some $c,c'>0$ we have
\begin{align*}
\sum _{n=1}^\infty \frac{1}{a_n^{1-\alpha}}&=\sum _{n=1}^\infty \sum_{k\in A_n}\frac{1}{k^{1-\alpha }}\\
&\geq c\sum _{n=1}^\infty \frac{2^n}{n^{\gamma }}\frac{2^{n(\alpha -1)}}{(1+1/n^{\gamma })^{1-\alpha}}
\geq c'\sum _{n=1}^\infty \frac{2^{n\alpha}}{n^{\gamma}}=\infty .
\end{align*}
Using (\ref{eq123}) we obtain
\begin{align*}
\kap (A_n)&\leq \frac{|A_n|}{\min _{a\in A_n}\sum _{b\in A_n}G_\alpha (a,b)}\leq c_1 \frac{2^n}{n^\gamma }\Big( \sum _{k=1}^{[\frac{2^{n}}{n^{\gamma }}]}\frac{1}{k^{1-\alpha}} \Big)^{-1}\\
&\leq c_2 \frac{2^n}{n^\gamma }\Big(\int _1^{\frac{2^n}{n^{\gamma }}}\frac{\ud x}{x^{1-\alpha}}\Big)^{-1}
\leq c_3\frac{2^{n(1-\alpha)}}{n^2},
\end{align*}
for some $c_1,c_2,c_3>0$.
It follows that
\begin{align*}
\sum _{n=1}^\infty \frac{\kap (A_n)}{2^{n(1-\alpha)}}\leq c\sum _{n=1}^\infty n^{-2}<\infty . 
\end{align*}
By the test of massiveness the set $A$ is not massive.
\end{example}
Here are interesting examples of sequences where Theorem \ref{thmSuperlin} apply. Let $a_n=[h(n)]$, where $h$ is a regularly varying function of index $\beta > 0$ such that $h\in C^1$ in some neighbourhood of infinity and
\begin{align*}
xh^{\prime}(x)/h(x)\rightarrow \beta ,\quad \textrm{as}\ x\rightarrow \infty .
\end{align*}
For instance, one can choose the following functions:
\begin{align*}
f(x)=x^\beta,\qquad f(x)=x^\beta \log ^\gamma (x),\qquad f(x)=x^\beta \exp(a\log ^\gamma x),\ 0<\gamma <1 .
\end{align*}
We claim that $\mathcal{A}=\{ a_n\}$ is $S_\alpha$-massive if and only if $\beta \leq 1/(1-\alpha)$.  
\begin{itemize}
\item[(i)] Assume that $\beta > 1$. Let us show that $\{a_n\}$ satisfies the condition $a_n\geq a_{n-k}+a_k$ for $k$ and $n-k$ large enough. The function $h(x)$ can be represented in the form $h(x)=\big(x v (x)\big)^{\beta ^*}$, where $v(x)$ is eventually non-decreasing and $1<\beta ^*< \beta $. Indeed, for an appropriately chosen slowly varying function $l$ we have
\begin{align*}
h(x)=x^\beta l(x)=x^{\beta ^*}\big( x^{\frac{\beta}{\beta ^*}-1}l^{\frac{1}{\beta ^*}}(x)\big)^{\beta ^*}
=x^{\beta ^*}v^{\beta ^*}(x).
\end{align*}
Let us show that $v(x)=\big(x^{-\beta ^*}h(x)\big)^{1/\beta ^*}$ is eventually increasing. Indeed, taking derivative we obtain
\begin{align*}
v^{\prime}(x)=x^{-2}h^{1/\beta ^{*}}(x)\Big(\frac{1}{\beta ^{*}}\frac{xh^{\prime}(x)}{h(x)}-1\Big)>0,
\end{align*}
for $x>k_0>1$ large enough.
For $k$ and $n $ such that $\min \{k,n-k\}>k_0$, we have
\begin{align*}
a_{n-k}+a_{k}&\leq \Big[(n-k)^{\beta ^*} v ^{\beta ^*} (n-k) + k^{\beta ^*} v ^{\beta ^*} (k)\Big]\\
&\leq \Big[\big((n-k)v(n-k)+kv(k)\big)^{\beta ^*} \Big]\\
&\leq [\big(nv(n)\big)^{\beta ^*} ]=a_n.
\end{align*}
Thus $\mathcal{A}=\{a_n\}$ is eventually superlinear. We apply Theorem \ref{thmSuperlin} to conclude that the set $\mathcal{A}$ is $S_\alpha$-massive if and only if $\beta \leq 1/(1-\alpha )$.
\item[(ii)] Assume now that $0<\beta \leq 1$. In this case $\mathcal{A}$ is $S_\alpha$-massive. Indeed, we apply the test of massiveness. Let $h(x)=x^\beta l^\beta (x)$, where $l$ is a slowly varying function. Let $\pi _\mathcal{A}(x)$ be the distribution function of the sequence $\mathcal{A}$. According to \cite[Proposition 1.5.15]{Bingham}, 
\begin{align*}
\pi _\mathcal{A}(x)\sim  x^{1/\beta}l^{\#}(x^{1/\beta}),
\end{align*}  
where $l^{\#}$ is the de Bruijn conjugate of the function $l$, that is, the slowly varying function which is unique up to asymptotic equivalence and satisfies  
\begin{align*}
l(x)l^{\#}(xl(x))\rightarrow 1,\quad l^{\#}(x)l(xl^{\#}(x))\rightarrow 1,\quad \textrm{as}\ \  x\rightarrow \infty .
\end{align*} 
Set $\mathcal{A}_n=\mathcal{A}\cap [2^n,2^{n+1})$. Applying the inequality (\ref{isoperim}) we get
\begin{align*}
\kap (\mathcal{A}_n)\geq c_1|\mathcal{A}_n|^{1-\alpha }\geq c_2(\pi _\mathcal{A}(2^n))^{1-\alpha},
\end{align*}
for some $c_1,c_2>0$. Hence
\begin{align*}
\sum _{n=1}^\infty \frac{\kap (\mathcal{A}_n)}{2^{n(1-\alpha)}}\geq c_3\sum _{n=1}^\infty 2^{n(1-\alpha)(1/\beta -1)}(l^{\#}(2^{n/\beta})^{1-\alpha}=\infty ,
\end{align*}
for some $c_3>0$. The claim follows.
\end{itemize}

\par We do not know whether the set $\mathcal{P}=\{p_n:\, n\in \gN \}$ (or the resulting set after omitting finitely many terms) of successive primes is superlinear that is, for some $n_0\in \gN$,  $p_{n-n_0}\geq  p_{n-k-n_0}+p_{k-n_0}$, for all $n_0<k<n$. 
What is true is that the successive differences $\Delta p_n= p_{n}-p_{n-1}$ do not form an increasing sequence. Indeed, it is easy to show that for any $A>0$ the set of $p_n\in \mathcal{P}$ such that $p_{n+1}-p_n\geq A$ is infinite. On the other hand, according to the recent paper of Zhang \cite{Zhang} the set of $n$ such that $p_{n+1}-p_n< B$ is also infinite for some $B>0$. Thus we cannot rely on Corollary \ref{convex}.
\par In order to study $S_\alpha$-massiveness of the set of prime numbers or its subsets we apply the test of massiveness. We owe to say that our work was strongly inspired by the paper of McKean \cite{McKean} about massiveness of the set of primes with respect to the simple random walk in $\gZ ^3$.  
\begin{thm}\label{Primes}
The set of primes $\cP$ is $S_\alpha$-massive for all $0<\alpha <1$.
\end{thm}
\begin{proof}
Let $\mathcal{P}_n = \mathcal{P}\cap [2^n, 2^{n+1})$. Let $\pi (x)$ be the density function of the sequence $\{p_n\}$ of successive primes. By the Prime Number Theorem,
\begin{align*}
\pi (x)\sim \frac{x}{\log x},\quad \textrm{at}\ \infty ,
\end{align*}
whence
\begin{align}
c_1\frac{2^n}{ n}\leq |\mathcal{P}_n|\leq c_2\frac{2^n}{ n},\label{piBounds}
\end{align}
for some constants $c_1,c_2>0$.
The inequality (\ref{isoperim}) yields
\begin{align*}
\kap (\mathcal{P}_n)&\geq c_3|\mathcal{P}_n|^{1-\alpha }
\geq c_4\Big(\frac{2^n}{n}\Big)^{1-\alpha},
\end{align*}
for some $c_3,c_4>0$.
Finally applying the test of massiveness we obtain
\begin{align*}
\sum _{n>1}\frac{\kap (\mathcal{P}_n)}{2^{n(1-\alpha)}}\geq c\sum _{n>1}\frac{1}{n^{1-\alpha }}=\infty
\end{align*}
for some $c>0$. The proof is finished.
\end{proof}
Examples below show that there are non-trivial proper subsets of $\mathcal{P}$ which are $S_\alpha$-massive.
\begin{example}
Let $h$ be a smoothly varying function of index $\beta >0$. Assume that $h$ satisfies the conditions from the paper \textsc{Leitmann} \cite{leitman} (1977) (see also recent paper \textsc{Mirek} \cite{mirek}). For instance, $h$ is one of the following functions
\begin{align*}
h(x)=x^\beta,\qquad h(x)=x^\beta \log ^\gamma (x),\qquad h(x)=x^\beta \exp(a\log ^\gamma x),\ 0<\gamma <1. 
\end{align*}
Let $\mathcal{P}_h$ be the set of primes of the form $p=[h(n)]$ and 
\begin{align*}
\pi _h(x)=\#\{p\in \cP _h:\, p\leq x\}.
\end{align*}
It was shown in \cite{leitman} that
\begin{align*}
\pi _h(x)\sim \frac{\phi (x)}{\log x}\qquad \textrm{at}\ \ \infty ,
\end{align*} 
where $\phi $ is inverse of $h$ and $\beta \in [1,\frac{12}{11})$. 
\par Among the variety of the classes $\cP _h$ we would like to mention the class $\cP _\beta $ of Piatetski-Shapiro primes \textsc{Piatetski-Shapiro} \cite{piatetski} (1953) which is defined by the function $h(x)=x^{\beta }$. It has been recently proved in \textsc{Rivat and Sargos} \cite{sargos} that for the class $\cP _\beta $ the interval $[1,\frac{12}{11})$ can be enlarged to $[1,\frac{2817}{2426})$. 
\par \textbf{Claim.} \textit{For every $\alpha \in (0,\frac{391}{2817})$, the set  $\mathcal{P}_\beta $ is not $S_\alpha$-massive.}\\
\vspace*{0,1cm}
Indeed, by Theorem \ref{thmSuperlin} it is enough to check that 
\begin{align*}
\sum _{n\in \cP _\beta }\frac{1}{n^{1-\alpha}}<\infty .
\end{align*}
Let $\beta \in [1,\frac{2817}{2426})$ and $\alpha +1/\beta <1$. 
We have
\begin{align*}
\sum _{n\in \cP _\beta }\frac{1}{n^{1-\alpha}}= \int _2^\infty \frac{\ud \pi _\beta (x)}{x^{1-\alpha}}=-\frac{1}{2^{1-\alpha}}+(1-\alpha)\int _2^\infty \frac{ \pi _\beta (x)\ud x}{x^{2-\alpha}}\\
\end{align*}
and, since  $\pi _\beta (x)\sim x^{1/\beta}/\log x$ at infinity,
\begin{align*}
\int _2^\infty \frac{ \pi _\beta (x)\ud x}{x^{2-\alpha}} \asymp \int _2^\infty \frac{\ud x}{x^{2-\alpha -1/\beta}\log x}<\infty .
\end{align*}
The claim follows.
\par On the other hand, choose $h(x)=x\log ^Cx$, $C>0$. Let us show that the set $\cP _h$ is $S_\alpha$-massive for all $\alpha \geq C/(1+C) $. Indeed, we have
\begin{align*}
\pi _h(x)\sim \frac{x}{\log ^{1+C}(x)}\quad \textrm{at}\ \infty .
\end{align*}
As in the proof of Theorem \ref{Primes} we obtain
\begin{align*}
\kap (\cP _{h}\cap [2^n,2^{n+1}))\geq c_1\big(\pi _h(2^n)\big)^{1-\alpha }\geq c_2\frac{2^{n(1-\alpha )}}{n^{(1+C)(1-\alpha)}}.
\end{align*}
At last the test of massiveness yields the result.
\end{example}
\section{Massive subsets of $\gZ ^2$}
Let $S_\alpha $ be the subordinated random walk in $\mathbb{Z} ^2$ as defined in Introduction. $S_\alpha$ is transient for all $0< \alpha <2$. According to \cite[Theorem 2.4]{Cygan} its Green function satisfies 
\begin{align}
G_\alpha (x,y)\sim  \frac{\Gamma \big( 1-\frac{\alpha}{2}\big)}{\pi \Gamma \big( \frac{\alpha}{2}\big) }\, \Vert x-y\Vert ^{\alpha -2} .\label{Green2}
\end{align}
\begin{prop}\label{prop12}
Let $B=\gN \times \{0\}\subset \mathbb{Z} ^2$. The set $B$ is $S_\alpha$-massive if and only if $1\leq \alpha <2$.
\end{prop}
\begin{proof}
We use the inequality (\ref{eq123}) and the test of massiveness. Let $A_n=B\cap \{x\in \gZ ^2: 2^n\leq \Vert x\Vert _{\infty }<2^{n+1}\}$ then
\begin{align*}
\max _{x\in A_n}\sum _{y\in A_n}G_{\alpha }(x,y)&\leq c_1 \Big(1+\!\!\! \!\sum _{2^n<y<2^{n+1}}\big|2^n-y\big|^{\alpha -2}\Big)\leq c_2\int _1^{2^n}t^{\alpha -2}\ud t\\
&\leq c_3\left\{ \begin{array}{ll}
1, & \alpha <1\\
n, & \alpha =1\\
2^{n(\alpha -1)}, & 1< \alpha < 2,
\end{array} \right.
\end{align*}
for some constants $c_1,c_2,c_3>0$. Similarly,
\begin{align*}
\min _{x\in A_n}\sum _{y\in A_n}G_{\alpha }(x,y)&\geq c_4 \Big(1+\!\!\! \!\sum _{3\cdot 2^{n-1}<y<2^{n+1}}\big|3\cdot 2^{n-1}-y\big|^{\alpha -2}\Big)\geq c_5\int _1^{2^{n-1}}t^{\alpha -2}\ud t\\
&\geq c_6\left\{ \begin{array}{ll}
1, & \alpha <1\\
n, & \alpha =1\\
2^{n(\alpha -1)}, & 1< \alpha < 2,
\end{array} \right.
\end{align*}
for some $c_4,c_5,c_6>0$. 
Thus we have
\begin{align*}
\frac{\kap(A_n)}{2^{n(2-\alpha)}}\asymp \left\{ \begin{array}{ll}
2^{n(\alpha -1)}, & 0<\alpha <1\\
n^{-1}, & \alpha =1\\
1, & 1< \alpha < 2.
\end{array} \right.
\end{align*}
Finally the test of massiveness yields the result.
\end{proof}
\par A set $\cA \subset \gZ ^2$ is called \textit{radially bounded} if there exists $N>0$ such that for any $r>0$
\begin{align*}
\#\{a\in \cA:\, \Vert a\Vert _\infty =r\}\leq N.
\end{align*}
A radially bounded set $\cA$ is called \textit{superlinear} if the set of numbers $\{\Vert x\Vert _\infty :\, x\in \cA \}$ enumerated in the increasing order is superlinear as defined at (\ref{superlin}).
\par Let $B=\gN \times \{0\} $. When $B$ is not $S_\alpha$-massive (i.e. $0<\alpha < 1$) none of the radially bounded sets is $S_\alpha$-massive whereas if $B$ is $S_\alpha$-massive (i.e. $1\leq \alpha <2$) among radially bounded sets there are $S_\alpha$-massive as well as non--$S_\alpha$-massive sets.  
\begin{thm}
Assume that $1\leq \alpha <2$ and that the set $\cA \subset \gZ ^2$ is superlinear. Then $\cA$ is $S_\alpha$-massive if and only if 
\begin{align*}
\sum _{a\in \cA} \frac{1}{\Vert a\Vert _\infty ^{2-\alpha}}=\infty .
\end{align*}
\end{thm}
Proof of this statement follows line by line the proof of Theorem \ref{thmSuperlin}. 
\par When $0<\alpha <1$ the set $B= \{0\}\times \gN $ is not $S_\alpha$-massive whereas any cone around $B$ is massive. We study massiveness of thorns $\mathcal{T}$ defined as,
\begin{align*}
\mathcal{T}=\{ (x_1,x_2)\in \mathbb{Z} ^2:\, |x_1|\leq t(x_2),\, x_2\geq 1 \},\quad \frac{t(n)}{n}=o(1).
\end{align*}
\begin{thm}
The thorn $\mathcal{T}$ as defined above is $S_\alpha $-massive if and only if
\begin{align*}
\sum _{n\geq 1}\Big(\frac{t(2^n)}{2^n}\Big)^{1-\alpha}=\infty .
\end{align*}
\end{thm}
\begin{flushleft}
The proof of this statement is similar to that of \cite[Theorem 4.4]{Cygan} which is related to the lattice $\gZ ^d$ with $d\geq 3$.
\end{flushleft}
\vspace{0,2cm}
\par Let $\cA$ be a subset of the set $B= \{0\} \times \gN$. Let $\cT $ be a thorn. We define a subthorn $\cT _{\cA}$ related to $\cA$ as
\begin{align*}
\cT _{\cA}=\mathcal{T}\cap \{(x_1,x_2)\in \mathbb{Z} ^2:\, x_2\in \cA\}.
\end{align*}
\begin{figure}[h]
		\begin{tikzpicture}[scale=0.8]
									\draw[densely dashed] (-5,4) to (5,8);
									\draw[densely dashed] (-5,4) to (5,0);
									\draw[densely dashed] (-5,6) to (-3,6) to (-3,2) to (-5,2);
									\draw[densely dashed] (-5,8) to (-1,8) to (-1,0) to (-5,0);
									\draw[densely dashed] (3,0) to (3,8);
									\draw[thick] (-4,4) to [out=28, in=185] (5,7);
									\draw[thick] (-4,4) to [out=332, in=175] (5,1);
									\begin{scope}
										\clip (-4,4) to [out=28, in=185] (5,7)-- (5,1) to [out=175, in=332] (-4,4);
										\clip (-4,3) rectangle (-3,5);
										\fill[color=black, opacity=0.1] (-4,3) rectangle (-3,5);
									\end{scope}
									\begin{scope}
										\clip (-4,4) to [out=28, in=185] (5,7)-- (5,1) to [out=175, in=332] (-4,4);
										\clip (-3,2) rectangle (-1,6);
										\fill[color=black, opacity=0.2] (-3,2) rectangle (-1,6);
									\end{scope}	
									\begin{scope}
										\clip (-4,4) to [out=28, in=185] (5,7)-- (5,1) to [out=175, in=332] (-4,4);
										\clip (-1,1) rectangle (3,7);
										\fill[color=black, opacity=0.4] (-1,1) rectangle (3,7);
									\end{scope}	
									\begin{scope}
										\clip (-4,4) to [out=28, in=185] (5,7)-- (5,1) to [out=175, in=332] (-4,4);
										\clip (3,0) rectangle (5,8);
										\fill[color=black, opacity=0.5] (3,0) rectangle (5,8);;
									\end{scope}	
									\draw[thick] (-5,4) to (5,4);
									\draw (-5.8,3.65) node [right] { {\tiny $(0,0)$}};
									\draw (-3,2) node [right] { {\tiny $2$}};
									\draw (-1,0) node [right] { {\tiny $4$}};
									\draw (3,0) node [right] {{\tiny $8$}};
									\draw (4.2,7.3) node [right] {{\tiny $t(n)$}};
									\draw [fill=black, black] (-2,4) circle [radius=0.1];
									\draw (-2,3.8) node [right] { {\tiny $a_1$}};
									\draw [fill=black, black] (-2,5) circle [radius=0.1];
									\draw [fill=black, black] (-2,3) circle [radius=0.1];
									\draw [fill=black, black] (1,4) circle [radius=0.1];
									\draw (1,3.8) node [right] {{\tiny $a_2$}};
									\draw [fill=black, black] (1,5) circle [radius=0.1];
									\draw [fill=black, black] (1,3) circle [radius=0.1];
									\draw [fill=black, black] (1,2) circle [radius=0.1];
									\draw [fill=black, black] (1,6) circle [radius=0.1];
									\draw [fill=black, black] (2,4) circle [radius=0.1];
									\draw (2,3.8) node [right] {{\tiny $a_3$}};
									\draw [fill=black, black] (2,3) circle [radius=0.1];
									\draw [fill=black, black] (2,5) circle [radius=0.1];
									\draw [fill=black, black] (2,6) circle [radius=0.1];
									\draw [fill=black, black] (2,2) circle [radius=0.1];
									\draw [fill=black, black] (4,4) circle [radius=0.1];
									\draw (4,3.8) node [right] {{\tiny $a_4$}};
									\draw [fill=black, black] (4,3) circle [radius=0.1];
									\draw [fill=black, black] (4,5) circle [radius=0.1];
									\draw [fill=black, black] (4,6) circle [radius=0.1];
									\draw [fill=black, black] (4,2) circle [radius=0.1];
									\draw [fill=black, black] (4,6.8) circle [radius=0.1];
									\draw [fill=black, black] (4,1.2) circle [radius=0.1];
				\end{tikzpicture}
							\caption{The subthorn.}
\end{figure}

Let us recall that a positive function $f$ defined on $[0,\infty )$ is called doubling if there exist a constant $C>0$ and some number $x_0>0$ such that 
\begin{align*}
f(2x)\leq Cf(x),\quad x\geq x_0.
\end{align*}
\begin{prop}\label{prop1}
Let $0<\alpha <1$. Assume that $t(n)=n/L(n)$ 
with $L(n)\rightarrow \infty$. Let $\pi _{\cA}=x/l(x)$ be the density function of the set $\cA$, 
where $l(n)\rightarrow \infty$. Assume that both $t(n)$ and $\pi _{\cA}(n)$ are doubling. Then the subthorn $\cT _{\cA}$ is $S_\alpha$-massive if
\begin{align}
\sum _{n=1}^{\infty} \Big(L(2^n)\, l(2^n)\Big)^{\frac{\alpha}{2}-1}=\infty .\label{series123}
\end{align}
\end{prop}
\begin{proof}
We define the subthorns
\begin{align*}
\cT _{\cA _n}=\cT _{\cA}\cap \{x\in \gZ ^2: 2^n\leq \Vert x\Vert _\infty <2^{n+1}\}
\end{align*}
and apply the test of massiveness. We have
\begin{align*}
|\cT _{\cA _n}|=\int _{2^n}^{2^{n+1}}t(x)\ud \pi _\cA (x)\asymp t(2^n)\pi _\cA (2^n),
\end{align*}
whence
\begin{align*}
\kap (\cT _{\cA _n})&\geq c_1|\cT _{\cA _n}|^{1-\alpha /2}\asymp \big(t(2^n)\pi _\cA (2^n)\big)^{1-\alpha /2}.
\end{align*}
At last the test of massiveness yields the result.
\end{proof}
\begin{example}
1. Assume that $t(n)=n/\log \log n$ and $\cA$ be the set of primes $\cP$. By the Proposition \ref{prop1} the subthorn $\cT _{\cP}$ is $S_\alpha$-massive in $\gZ ^2$ for all $0<\alpha <1$. Indeed, we have
\begin{align*}
L(2^n)\, l(2^n)\asymp n\log n
\end{align*}
and since $1-\alpha /2 <1$ the series in (\ref{series123}) diverges.\\
\vspace*{0,2cm}
2. Assume that $t(n)=n/(\log n)^\gamma$ for some $0<\gamma <1$ and let $\cA=\cP$ as above. Let us show that for all $\alpha  \in I=[\frac{2\gamma}{1+\gamma},1)$ the subthorn $\cT _{\cP}$ is $S_\alpha$-massive. Observe that by our assumption the interval $I$ is not empty. We have 
\begin{align*}
L(2^n)\, l(2^n)\asymp n^{1+\gamma}.
\end{align*}
Since $0<\gamma <1$ the series in (\ref{series123}) diverges.
\end{example}
\par Whether $\cT _\cP$ is non-massive for all/some $\alpha \in (0, \frac{2\gamma}{1+\gamma})$ is an open question at the present writing.
\section{Capacity bounds}
The aim of this section is to comment on a general capacity lower bound of the form
\begin{align}
\mathrm{Cap} (B)\geq m (B)N\Big( \frac{1}{m (B)}\Big).\label{iso2}
\end{align}
Various versions of this inequality have been used repeatedly in the previous two sections of this work and in the recent paper \cite{Cygan}. 
\par  
Let $(X, d, m )$ be a locally compact separable metric measure space equipped with a regular Dirichlet form $(\cE , \cF)$, $\cF \subset L_2(m )$.
Consider a Nash-type inequality
\begin{align}
\cE (f,f)\geq \Vert f\Vert ^2_2\, N (\Vert f\Vert ^2_2),\qquad \Vert f\Vert _1=1, \label{Nash}
\end{align}
which can be true or not for some function $N$ satisfying $N(x)=o(x)$ at infinity.
\begin{thm}\label{capbound}
Let $\cN$ and $\cM$ be two complementary Young functions. Assume that the Dirichlet form $(\cE , \cF)$ is transient and that $\cM$ is doubling. Then (\ref{Nash}) implies (\ref{iso2}) with $N=\cN ^{-1}$.
\end{thm}
\begin{flushleft}
The proof of Theorem \ref{capbound} is based on two crucial ingredients:
\end{flushleft}
\begin{flushleft}
1. The inequality (\ref{Nash}) implies that
\begin{align}
\cE (f,f)\geq \Vert f^2\Vert _{\cM},\label{orlicz}
\end{align}
where $\Vert f^2\Vert _{\cM}$ denotes the Orlicz norm of the function $f$ associated with the Young function $\cM$, see \textsc{Bakry, Coulhon, Ledoux and Saloff-Coste} \cite[Theorem 10.5]{sobolev}.
\end{flushleft}
\begin{flushleft}
2. The inequalities (\ref{iso2}) and (\ref{orlicz}) are equivalent statements, 
see \textsc{Kaimanovich} \cite[Theorem 3.1]{Kaimanovich}.
\end{flushleft}
\begin{corollary}\label{CorPsi}
In the setting of Theorem \ref{capbound}, let $\psi$ be a regularly varying Bernstein function. Assume that it has the same index $\beta $ at $0$ and at $\infty$. Assume that $\cE  (f,f)=(\cL f,f)$ satisfies the following Nash-type inequality
\begin{align}
\cE  (f,f)\geq \Vert f\Vert ^{2+\frac{4}{d}}_2,\qquad \Vert f\Vert _1=1,\label{clasNash}
\end{align}  
with $d\geq 1$.
Suppose that $0<\beta < d/2$, then the Dirichlet form $\cE _\psi (f,f)=(\psi(\cL ) f,f)$ is transient. Let $\mathrm{Cap} _\psi (B)$ be the capacity of $B$ associated with $\cE _\psi (f,f)$. Then
\begin{align}
\mathrm{Cap}_\psi (B)\geq m(B)\psi \Big( \frac{1}{m(B)^{2/d}} \Big).\label{CapPsi}
\end{align}
In particular, choosing $\psi (s)=s^{\alpha /2}$, $0<\alpha <\min \{2, d\}$, we obtain
\begin{align*}
\mathrm{Cap}_\psi (B)\geq m(B)^{1-\alpha /d}.
\end{align*}
\end{corollary}
\begin{proof}
According to \cite{Maheux} (see also \cite{schilling1}) the inequality (\ref{clasNash}) yields
\begin{align*}
\cE _\psi (f,f)\geq \Vert f\Vert _2^{2}\psi (\Vert f\Vert ^{2/d}_2),\qquad \Vert f\Vert _1=1
\end{align*}  
and Theorem \ref{capbound} gives the result. Indeed, in our case the function $\cN (t) =\big(\psi ^{-1}(t)\big)^{d/2}$ is regularly varying of index $d/2\beta >1$. Hence it is a Young function whose complementary function $\cM$ is regularly varying of index $\gamma $,
\begin{align*}
\frac{1}{\gamma} + \frac{2\beta}{d} =1. 
\end{align*}
For all of this see \cite[Section 1.8.4, Theorem 1.8.10]{Bingham}. In particular, $\cM$ is doubling. At last the transience of $\cE _\psi (f,f)$ follows by the ultracontractivity theorem of Coulhon \cite[Proposition II.1]{coulhon}. The proof is finished.
\end{proof}
\textbf{Acknowledgements}\\
This work has been started while visiting the Bielefeld University. We would like to thank  A.~Grigor'yan, W.~Hansen, M.~Mirek, S.~Molchanov and B.~Trojan for fruitful discussions and valuable comments.
\bibliography{biblio}
\bibliographystyle{plain}
\end{document}